\documentclass[10pt,a4wide]{article}
\usepackage{amsmath} % assumes amsmath package installed
\usepackage{amssymb}  % assumes amsmath package installed
\usepackage{amsthm}
\usepackage{caption}
\usepackage{subcaption}
\usepackage{dsfont}
\usepackage{bbm}
\usepackage{graphicx}
\usepackage{psfrag}
\usepackage{hyperref}
\usepackage{color}
\usepackage{mathdots}
\usepackage{pst-all}
\usepackage{pstricks}
\usepackage{pstricks-add}
\usepackage{pst-plot}
\usepackage{pst-infixplot}
\usepackage{pstricks-add}
\usepackage{wrapfig}
\usepackage{tikz}
\usepackage{capt-of}

\usepackage{mathtools}
\usepackage[hyphenbreaks]{breakurl}
\usepackage{tikz}
\usetikzlibrary{shapes,arrows}
\usetikzlibrary{arrows,calc}

\tikzset{
    block/.style = {draw, rectangle, 
        minimum height=1cm, 
        minimum width=2cm},
    input/.style = {coordinate,node distance=1cm},
    output/.style = {coordinate,node distance=4cm},
    arrow/.style={draw, -latex,node distance=2cm},
    pinstyle/.style = {pin edge={latex-, black,node distance=2cm}},
    sum/.style = {draw, circle, node distance=1cm}
}

\newtheorem{theorem}{Theorem}[section]
\newtheorem{lemma}[theorem]{Lemma}
\newtheorem{proposition}[theorem]{Proposition}

\newtheorem{definition}[theorem]{Definition}
 \newtheorem{remark}[theorem]{Remark}

\renewcommand{\phi}{\varphi}		

\makeatletter
\renewcommand*{\@fnsymbol}[1]{\ifcase#1\or*\else\@arabic{\numexpr#1-1\relax}\fi}
\makeatother

\title
{On exact controllability of infinite-dimensional\\ linear port-Hamiltonian systems\thanks {
Support by Deutsche Forschungsgemeinschaft (Grant JA 735/13-1) 
is gratefully acknowledged.}}

\author{Birgit Jacob\thanks{University of Wuppertal, School of Mathematics and Natural Sciences,
Gau\ss stra\ss e 20,
D-42119 Wuppertal, Germany, $\{$bjacob,julia.kaiser$\}$@uni-wuppertal.de}
\and Julia T.~Kaiser\footnotemark[2]}
\date{}

\begin{document}
\maketitle
\thispagestyle{empty}
\pagestyle{empty}

\begin{abstract}
Infinite-dimensional linear port-Hamiltonian systems on a one-dimensional spatial domain with  full boundary 
control and without internal damping are studied. This
class of systems includes models of beams and waves as well as the transport equation and networks of nonhomogeneous 
transmission lines. The main result shows that well-posed port-Hamiltonian systems, with state space 
$L^2((0,1);\mathbb C^n)$ and input space $\mathbb C^n$, are exactly controllable. 
\end{abstract}


\section{Introduction}\label{intro}

In this article, we consider infinite-dimensional linear port-Hamiltonian systems on a one-dimensional 
spatial domain with boundary control of the form 
\begin{align}
\frac{\partial x}{\partial t}(\zeta,t) =& \left( P_1
  \frac{\partial}{\partial \zeta} + P_0\right) ({\cal H}({\zeta})
x(\zeta,t)),\nonumber \\
x(\zeta,0) =& x_0(\zeta),\label{eqn:pde}\\
u(t)=& \widetilde{W}_B \begin{bmatrix} ({\cal H}x)(1,t) \\ ({\cal
         H}x)(0,t)\end{bmatrix}, \nonumber
\end{align}
where $\,\, \zeta \in [0,1]$ and $t\ge 0$. Moreover, we assume that  $P_1$ is an invertible $n\times n$ Hermitian matrix, 
$P_0$ is a  $n\times n$  skew-adjoint matrix, $\widetilde{W}_B$ is a full row rank $n\times 2n$-matrix, 
and ${\cal H}(\zeta)$ is a positive $n\times n$ Hermitian matrix for a.e.~$\zeta\in (0,1)$  satisfying ${\cal H}, {\cal 
H}^{-1}\in L^{\infty}((0,1);\mathbb C^{n\times n})$. 
The matrix $P_1\mathcal{H}(\zeta)$ can be diagonalized as $P_1\mathcal{H}(\zeta)=S^{-1}(\zeta)\Delta(\zeta)S(\zeta)$, where 
$\Delta(\zeta)$ is a diagonal matrix and  $S(\zeta)$ is an invertible matrix  for a.e.~$\zeta\in (0,1)$. We suppose the 
technical assumption that  $S^{-1}$, $S$,  $\Delta: [0,1] \rightarrow \mathbb{C}^{n\times 
n}$ are continuously differentiable.

Equation \eqref{eqn:pde} describes a special class of port-Hamiltonian systems, which however is rich enough 
to cover in particular the wave equation, the transport equation and the Timoshenko beam 
equation, and also coupled beam and wave equations each with possibly damping on the boundary.
For more information on this class of port-Hamiltonian systems we refer to the monograph \cite{JacZwa12} and 
the survey \cite{JacZwa19}. However, we note that here we always assume that there is no internal damping (the matrix $P_0$ 
is skew-adjoint) and that we have full boundary control ($\widetilde{W}_B$ is a full row rank $n\times 2n$-matrix).

Port-based network modeling of complex physical systems leads to port-Hamiltonian systems. For 
finite-dimensional systems there is by now a well-established theory
\cite{vanDerSchaft06,EbMS07,DuinMacc09}. The port-Hamiltonian approach has been extended to the  
infinite-dimensional situation by a geometric differential 
approach \cite{VanDerSchaftMaschke_2002,MM05,JeSc09,ZwaGorMasVil10} and by a 
functional analytic approach \cite{Villegas_2007,ZwaGorMasVil10,JacZwa12,JacMorZwa15,Augner16,JacZwa19}. Here we follow the 
functional analytic point of view. This approach has been successfully used to derive simple verifiable conditions for 
well-posedness \cite{GorZwaMas05,Villegas_2007,ZwaGorMasVil10,JacZwa12,JacMorZwa15,JacKai}, stability 
\cite{JacZwa12,AugJac14} and stabilization \cite{RaZwGo,RaGoMaZw,AugJac14,SchZwa18} and robust regulation \cite{HuPa}. For 
example, the port-Hamiltonian system \eqref{eqn:pde} is well-posed, if 
$v^* P_1 v -w^* P_1 w\le 0$ for every $\left[\begin{smallmatrix} v\\w\end{smallmatrix}\right]\in \ker \widetilde{W}_B$.

Provided the port-Hamiltonian system \eqref{eqn:pde} is well-posed,  we aim  to characterize {\em exact controllability}. 
Exact controllability is a desirable property of a controlled partial differential equation and has been extensively 
studied, see for example \cite{Komo94,CurZwa95,TucsWeis09}. We call the port-Hamiltonian system {exactly controllable}, if 
every 
state of the system can be reached in finite time with a suitable control input. Triggiani \cite{Trig91} showed 
that 
exact controllability does not hold for many hyperbolic partial differential equations. However, in this paper  we prove, 
that the port-Hamiltonian system  \eqref{eqn:pde} is exactly
controllable whenever it is well-posed.

\section{Reminder on port-Hamiltonian systems}\label{section2}

We define
\begin{equation}\label{operatorA}
\frak Ax:= \left( P_1 \frac{d}{d\zeta} + P_0\right) ({\cal H}x), \qquad x\in {\cal D}(\frak A), 
\end{equation}
 on $X:=L^2((0,1);\mathbb{C}^n)$  with the domain
\begin{equation}
\label{domainA}
{\cal D}(\frak A) := \left\{ x\in X\mid  {\cal H}x\in H^{1}((0,1);\mathbb C^n) \right\}
\end{equation}
and $\frak B:{\cal D}(\frak A)\rightarrow \mathbb C^n$ by
\begin{equation}\label{operatorB}
\frak B x=\widetilde{W}_B \begin{bmatrix} ({\cal H}x)(1,t) \\ ({\cal
         H}x)(0,t)\end{bmatrix}.
\end{equation}
Here $H^{1}((0,1);\mathbb C^n)$ denotes the first order Sobolev space.
We call $\frak A$ the \emph{(maximal) port-Hamiltonian operator} and equip the state space  $X=L^2((0,1);\mathbb{C}^n)$ 
with 
the 
energy norm $\sqrt{\langle \cdot,{\cal H} \cdot\rangle}$, where $\langle\cdot,\cdot\rangle$ denotes the standard inner 
product on $L^2((0,1);\mathbb{C}^n)$. We note that the energy norm is equivalent to the standard norm on  
$L^2((0,1);\mathbb{C}^n)$.

Then the partial differential equation (\ref{eqn:pde})  can be written as a {\em boundary control system}
\begin{align*}
   \dot{x} (t) &= \frak A  x(t),\qquad   x(0)= x_0, \\
   u(t) &= \frak B x(t).
\end{align*}
%We refer the reader for the precise definition of a boundary control system to Section \ref{section2}. 
The first important question  is whether the port-Hamiltonian system \eqref{eqn:pde} is {\em well-posed} in the sense that 
for every initial condition $x_0\in X$ and every $u\in L^2_{\rm loc}([0,\infty);\mathbb C^n)$ equation \eqref{eqn:pde} has 
a unique mild solution. 
%Again, for the precise definition of well-posedness and mild solutions we refer to Section 
%\ref{section2}. 

In \cite{Villegas_2007,ZwaGorMasVil10,JacZwa12} it is shown that 
the port-Hamiltonian system \eqref{eqn:pde} is well-posed if and only if the operator $A:{\cal D}(A)\subset X\rightarrow 
X$, 
defined by
\begin{equation}\label{operatorA2}
Ax:= \left( P_1 \frac{d}{d\zeta} + P_0\right) ({\cal H}x), \qquad x\in {\cal D}(A), 
\end{equation}
with the domain
\begin{equation}
\label{domainA2}
{\cal D}(A) := \left\{ x\in {\cal D}(\frak A)\mid  \widetilde{W}_B\begin{bmatrix} (\mathcal{H} x)(1) \\ (\mathcal{H} 
x)(0)\end{bmatrix} = 0 
\right\}
\end{equation}
generates a strongly continuous semigroup on $X$. 
We recall, that  $A$ generates a contraction semigroup  on $X$ if and only if $A$ is dissipative on $X$, c.f. 
\cite{GorZwaMas05,JacZwa12,AugJac14}. Further, 
matrix conditions to guarantee generation of a contraction semigroup have been obtained  in 
\cite{GorZwaMas05,JacZwa12,AugJac14} and matrix conditions for the generation of strongly continuous semigroups can be 
found in \cite{JacMorZwa15}. 

For the proof of the main theorem feedback techniques are needed and therefore we investigate port-Hamiltonian 
systems with boundary control and observations. These are systems of the form
\begin{align}
\frac{\partial x}{\partial t}(\zeta,t) =& \left( P_1
  \frac{\partial}{\partial \zeta} + P_0\right) ({\cal H}({\zeta})
x(\zeta,t)),\nonumber \\
x(\zeta,0) =& x_0(\zeta),\label{eqn:pdeio}\\
u(t)=& \widetilde{W}_B\begin{bmatrix} ({\cal H}x)(1,t) \\ ({\cal
         H}x)(0,t)\end{bmatrix}\, \nonumber\\
         y(t)=& \widetilde{W}_C\begin{bmatrix} ({\cal H}x)(1,t) \\ ({\cal
         H}x)(0,t)\end{bmatrix}, \nonumber
\end{align}
where we restrict ourselves in this article to case where $P_1$, $P_0$, ${\cal H}$ and $\widetilde W_B$ satisfy the 
condition described in Section \ref{intro} and $\widetilde W_C$ is a full row rank $k\times 2n$ matrix, 
$k\in\{0,\cdots,n\}$, such that the matrix $\left[\begin{smallmatrix} \widetilde W_B\\\widetilde 
W_C\end{smallmatrix}\right]$ has full row rank.
We call system \eqref{eqn:pdeio} a {\em (boundary control and observation) port-Hamiltonian system}. The case $k=0$ refers 
to the case of a system without observation, that is, every definition or statement of the port-Hamiltonian system  
\eqref{eqn:pdeio} also applies to the port-Hamiltonian system  \eqref{eqn:pde}.

We define $\frak C:{\cal D}(\frak A)\rightarrow \mathbb C^k$ by
\begin{equation}\label{operatorC}
\frak C x=\widetilde{W}_C\begin{bmatrix} ({\cal H}x)(1,t) \\ ({\cal
         H}x)(0,t)\end{bmatrix}.
\end{equation}

Then we can write the port-Hamiltonian system \eqref{eqn:pdeio} in the following form 
\begin{align}
   \dot{x} (t) &= \frak A  x(t),\qquad   x(0)= x_0, \nonumber\\
   u(t) &= \frak B x(t),\label{eqn:bcs}\\
	y(t)&=\frak C x(t).\nonumber
\end{align}
If the operator $A$, defined by \eqref{operatorA2}-\eqref{domainA2}, generates a strongly continuous semigroup on the state 
space $X$, then \eqref{eqn:bcs} defines a {\em boundary control and observation system}, see \cite[Theorem 11.3.2 and 
Theorem 11.3.5]{JacZwa12}.

\begin{definition}
 Let $ {\mathfrak A}: {\cal D}({\mathfrak A})\subset X\rightarrow X$,  ${\mathfrak B}:{\cal D}({\mathfrak A})\rightarrow 
\mathbb C^n$ and ${\mathfrak C}:{\cal D}({\mathfrak A})\rightarrow \mathbb C^k$ be linear  operators.
 Then $({\mathfrak A}, {\mathfrak B},{\mathfrak C})$ is a {\em boundary
control and observation system} if the following hold:\vspace{-0.8ex}
\begin{enumerate}
\item
  The operator $A:{\cal D}(A)\subset X \rightarrow X$ with ${\cal D}(A)={\cal D}({\mathfrak A} ) \cap 
\ker({\mathfrak B})$
  and
  $ Ax = {\mathfrak A}  x$ for $x\in {\cal D}(A)$
  is the infinitesimal generator of a strongly continuous semigroup  on $X$.
\item 
  There exists a right inverse $\widetilde B \in {\cal L}(\mathbb C^n,X)$ of ${\mathfrak B}$ in the sense  that for all $u 
\in
  \mathbb C^n$ we have $\widetilde Bu \in {\cal D}({\mathfrak A} )$,  $ {\mathfrak B} \widetilde B u = u$ and ${\mathfrak A} 
\widetilde B:\mathbb C^n\rightarrow X$ is   bounded.
  \item The operator ${\mathfrak C}$ is bounded from ${\cal D}(A)$ to $\mathbb C^k$, where ${\cal D}(A)$ is equipped with 
the graph norm of $A$.
  \end{enumerate}
\end{definition}

We recall, that if $A$, defined by \eqref{operatorA2}-\eqref{domainA2}, generates a strongly continuous semigroup on the 
state space $X$, then the port-Hamiltonian system \eqref{eqn:pdeio} is a boundary control and observation system. 

We note that for  $x_0\in {\cal D}({\mathfrak A})$ and $u\in C^2([0,\tau];\mathbb C^n)$, $\tau>0$, satisfying ${\mathfrak 
B}x_0=u(0)$, a boundary control and observation system $({\mathfrak A}, {\mathfrak B},{\mathfrak C})$ possesses a {\em 
unique classical solution} \cite[Lemma 13.1.5]{JacZwa12}.

For technical reasons we formulate the boundary conditions  equivalently via the boundary flow and the boundary effort.
As the matrix $ \left[\begin{smallmatrix} P_1 & -P_1\\ I & I\end{smallmatrix}\right]$ is invertible, we can write the 
port-Hamiltonian system \eqref{eqn:pdeio}
equivalently as
\begin{align}
\frac{\partial x}{\partial t}(\zeta,t) =& \left( P_1
  \frac{\partial}{\partial \zeta} + P_0\right) ({\cal H}({\zeta})
x(\zeta,t)),\nonumber \\
x(\zeta,0) =& x_0(\zeta),\label{eqn:pdeio2}\\
u(t)=& {W}_B\begin{bmatrix}
  f_{\delta, \mathcal{H} x}\\
  e_{\delta,\mathcal{H} x}
 \end{bmatrix}, \nonumber\\
         y(t)=& {W}_C\begin{bmatrix}
  f_{\delta, \mathcal{H} x}\\
  e_{\delta,\mathcal{H} x}
 \end{bmatrix}, \nonumber
\end{align}
where 
\begin{align*}
 \begin{bmatrix}
  f_{\delta, \mathcal{H} x}\\
  e_{\delta,\mathcal{H} x}
 \end{bmatrix}= \frac{1}{\sqrt{2}} \begin{bmatrix} P_1 & -P_1\\ I & I\end{bmatrix}\begin{bmatrix}(\mathcal{H} x)(1)\\ 
(\mathcal{H} x)(0) 
\end{bmatrix}
\end{align*}
and 
\begin{align}\label{operatorB2}
\widetilde W_B  =W_B \frac{1}{\sqrt{2}} \left[\begin{smallmatrix} P_1 & -P_1\\ I & I\end{smallmatrix}\right],\,
\widetilde W_C = W_C \frac{1}{\sqrt{2}}  \left[\begin{smallmatrix} P_1 & -P_1\\ I & I\end{smallmatrix}\right].
\end{align}
Here $f_{\delta, \mathcal{H} x}$ is called the  \emph{boundary flow}  and $e_{\delta,\mathcal{H} x}$ the 
\emph{boundary 
effort}.
The  port-Hamiltonian system \eqref{eqn:pdeio} is uniquely described by the tuple 
$(\frak A, \frak B, \frak C)$ given by \eqref{operatorA}, \eqref{domainA}, \eqref{operatorB} and \eqref{operatorC}.

Well-posedness is a fundamental property of boundary control and observation systems.

\begin{definition}
 We call a boundary control and observation system $({\mathfrak A}, {\mathfrak B},{\mathfrak C})$
  {\em well-posed} if 
  there exist a
  $\tau >0$ and $m_\tau\ge 0$ such that for all $x_0 \in {\cal D}({\mathfrak A})$ and $u \in C^2([0,\tau];\mathbb{C}^n)$ 
with $u(0)= {\mathfrak B} x_0$ the classical solution $x$, $y$ satisfy 
   \begin{align*}
    \|x(\tau)\|_{X}^2 &+ \int_0^{\tau} \|y(t)\|^2 dt \\\leq&
    m_\tau \left( \|x_0\|_{X}^2 + \int_0^{\tau} \|u(t)\|^2 dt  \right).
    \end{align*}
\end{definition}
\vspace{1ex}
There exists a rich literature on well-posed systems, see e.g.~Staffans \cite{Staf05} and Tuscnak and Weiss 
\cite{TucsWeis14}. 
In general, it is not easy to show that a boundary control and observation system is well-posed. However, for the   
port-Hamiltonian system \eqref{eqn:pdeio}
well-posedness is already satisfied if $A$ generates a strongly continuous semigroup.

\begin{theorem}\cite[Theorem 13.2.2]{JacZwa12}\label{wellposed}
The  port-Hamiltonian system \eqref{eqn:pdeio} is well-posed if and only if the operator $A$ defined by 
\eqref{operatorA2}-\eqref{domainA2} generates a strongly continuous semigroup on $X$.
\end{theorem}

There is a special class of  port-Hamiltonian systems for which well-posedness follows immediately. 

\begin{definition}
A    port-Hamiltonian systems \eqref{eqn:pdeio} is called 
{\em impedance passive}, if 
\begin{equation}\mathrm{Re\,}\langle\frak A x,x\rangle \leq \mathrm{Re\,}\langle\frak B x,\frak C 
x\rangle\label{ip}\end{equation} 
 for every $x\in {\cal D}(\frak A)$.
If we have equality in \eqref{ip},
then the port-Hamiltonian system is called {\em  impedance energy preserving}.
\end{definition}

The fact that a  port-Hamiltonian system is impedance energy preserving can be characterized by a easy 
checkable matrix condition.

\begin{theorem}\cite[Theorem 4.4]{GorZwaMas05}\label{iep}
The  port-Hamiltonian systems \eqref{eqn:pdeio} is {impedance energy preserving} if and only if
it 
holds 
\begin{equation}
\begin{bmatrix}
{W}_B\Sigma {W}_B^* &{W}_B\Sigma{W}_C^* \\
{W}_C\Sigma {W}_B^* & {W}_C\Sigma {W}_C^* 
\end{bmatrix}=
\begin{bmatrix}
0 & I \\ I & 0
\end{bmatrix},
\end{equation}
where $\Sigma =\left[\begin{smallmatrix} 0&I\\I&0\end{smallmatrix}\right]$.
\end{theorem}
\vspace{1ex}

\begin{remark}\label{rem}
 Every impedance energy preserving port-Hamiltonian system \eqref{eqn:pdeio} is well-posed; ${W}_B\Sigma {W}_B^* =0$ even
implies that $A$ generates a unitary strongly continuous group, c.f. \cite[Theorem 1.1]{JacMorZwa15}.
\end{remark}

In order to formulate the {\em mild solution} of a well-posed port-Hamiltonian system  \eqref{eqn:pdeio} we need to 
introduce some notation.
Let $X_{-1}$ be the completion of $X$ with respect to the norm 
$ \|x\|_{X_{-1}}= \|(\beta I -A)^{-1}x\|_X$ 
for some $\beta $ in the resolvent set $\rho(A)$ of $A$, this implies,
\[ 
  X\subset  X_{-1}
\]
and $X$ is continuously embedded and dense in $X_{-1}$. Furthermore, let $(T(t))_{t\ge 0}$ be the strongly continuous 
semigroup generated by $A$. The semigroup  $(T(t))_{t\ge 0}$ extends uniquely to a strongly continuous semigroup 
$(T_{-1}(t))_{t\ge 0}$ on $X_{-1}$ whose generator $A_{-1}$, with domain equal to $X$, is an extension of $A$, see e.g.\ 
\cite{EngNag00}. Moreover, we can identify $X_{-1}$ with the dual space of ${\cal D}(A^*)$ with respect to the pivot space 
$X$, see \cite{TucsWeis09}, that is $X_{-1}={\cal D}(A^*)'$.
If the  port-Hamiltonian system  \eqref{eqn:pdeio} is well-posed, then the unique mild solution is given by
\begin{equation*}
x(t)=T(t)x_0+ \int_0^t T_{-1}(t-s)(\frak A \widetilde B-A_{-1}\widetilde B) u(s)\, ds.
\end{equation*}
 Here the operator $\widetilde B:\mathbb C^n\rightarrow  L^2((0,1);{\mathbb C}^n)$ can be defined as follows
\[ (\widetilde Bu)(\zeta):= ({\mathcal H}(\zeta))^{-1}\left( S_{1}{\zeta}+S_{2}(1-\zeta)\right)u,\]
where $S_{1}$ and $S_{2}$ are $n\times n$-matrices given by
\[ \begin{bmatrix} S_{1} \\  S_{2} \end{bmatrix}:=  \begin{bmatrix} P_1& -P_1\\I&I\end{bmatrix}^{-1}\widetilde 
W_B^*(\widetilde W_B\widetilde W_B^*)^{-1}. \]

For a well-posed  port-Hamiltonian system \eqref{eqn:pdeio} the {\em transfer function} is given by  \cite[Theorem 
12.1.3]{JacZwa12}
\begin{equation*}
 G(s)=\frak C(sI-A)^{-1}(\frak A\widetilde B-s\widetilde B)+\frak {C}\widetilde B, \qquad s\in \rho(A),
 \end{equation*}
 where $\rho(A)$ denotes the resolvent set of $A$. The transfer function is bounded on some right half plane and equals the 
Laplace transform of the mapping $u(\cdot)\mapsto y(\cdot)$ if $x_0=0$.

\begin{definition}\cite[Definition 13.1.11]{JacZwa12}
A well-posed port-Hamiltonian system \eqref{eqn:pdeio} with transfer function $G$ is called \emph{regular} if 
$\lim_{s\in\mathbb{R},s\to\infty} G(s)$ 
exists. In this case the
 \emph{feedthrough operator} $D$ is defined as
\begin{align*}D:= \lim_{s\in\mathbb{R},s\to\infty} G(s).
\end{align*}
\end{definition}
\vspace{1ex}

\begin{lemma}\cite[Lemma 13.2.22]{JacZwa12}\label{lem:regular}
Under the standing assumptions every well-posed port-Hamiltonian system \eqref{eqn:pdeio} 
is regular.
\end{lemma}

So far, we have only considered {\em open-loop system}, that is, the input $u(t)$ is 
independent of the output $y(t)$, see Figure \ref{figopen}. Systems, where input and output are connected via a feedback 
law 
\begin{equation}\label{law}
 u(t)=Fy(t)+v(t),
\end{equation}
are called {\em closed-loop systems}, see Figure \ref{figclosed}.  Here $F$ denotes the so called {\em feedback operator} 
and $v(t)$ the new input.

\tikzstyle{block} = [draw,  rectangle, 
    minimum height=3em, minimum width=6em]
    \tikzstyle{blocksmall} = [draw,  rectangle, 
    minimum height=1.5em, minimum width=3em]
\tikzstyle{sum} = [draw,  circle, node distance=1cm]
\tikzstyle{input} = [coordinate]
\tikzstyle{output} = [coordinate]
\tikzstyle{pinstyle} = [pin edge={to-,thin,black}]

\begin{figure}[ht]
    \begin{center}
\begin{tikzpicture}[auto, node distance=2.2cm,>=latex']
    \node [input, name=input] {};
    \node [block, right of=input] (system) {$(\frak A, \frak B, \frak C)$};
        \node [output, right of=system] (output) {};
 		\draw [->] (input) -- node {$u$} (system);
				\draw [->] (system) -- node {$y$} (output);
\end{tikzpicture}
\end{center}
\captionof{figure}{open-loop system $(\frak A, \frak B, \frak C)$}\label{figopen}
\label{ols}
\end{figure}

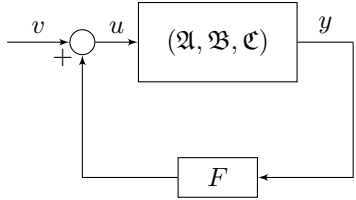
\begin{figure}[ht]
    \begin{center}
\begin{tikzpicture}[auto, node distance=1.8cm,>=latex']
    \node [input, name=input] {};
    \node [sum, right of=input] (sum) {};
    \node [block, right of=sum] (system) {$(\frak A, \frak B, \frak C)$};
       \node [output, right of=system] (output) {};
    \node [blocksmall, below of=system] (feedback) {$F$};
     \draw [->] (input) -- node {$v$} (sum);
		\draw [->] (sum) -- node {$u$} (system);
				\draw [-] (system) -- node {$y$} (output);
					    \draw [->] (output) |- node[pos=0.99] {} 
        node [near end] {} (feedback);
    \draw [->] (feedback) -| node[pos=0.99] {$+$} 
        node [near end] {} (sum);
\end{tikzpicture}
\end{center}
\captionof{figure}{closed-loop system $(\frak A, \frak B, \frak C)$ with feedback $F$}\label{figclosed}
\label{cls}
\end{figure}

\begin{definition}(\cite[Theorem 13.2.2]{JacZwa12} and \cite[Proposition 4.9]{We94})
%Let $(\frak A, \frak B, \frak C)$ be a well-posed boundary control and boundary observation port-Hamiltonian system  
\eqref{eqn:pdeio} and we denote by $D$ the corresponding feedthrough.
 A $n\times n$-matrix $F$ is called an {\em admissible 
feedback operator} for a regular  port-Hamiltonian system  \eqref{eqn:pdeio} with feedthrough operator $D$, if   $I-DF$ is 
invertible.
\end{definition}

\begin{proposition}\cite[Theorem 13.1.12]{JacZwa12}
Let $(\frak A, \frak B, \frak C)$ be a well-posed  port-Hamiltonian system  \eqref{eqn:pdeio}. Assume that
$F$ is an admissible feedback operator. Then the closed-loop system $(\frak A,  (\frak B-F\frak C), \frak C)$, i.e.,
\begin{align}
\frac{\partial x}{\partial t}(\zeta,t) =& \left( P_1
  \frac{\partial}{\partial \zeta} + P_0\right) ({\cal H}({\zeta})
x(\zeta,t)),\nonumber \\
x(\zeta,0) =& x_0(\zeta),\label{eqn:pdefb}\\
v(t)=& (\frak B-F\frak C)x(t) , \nonumber\\
         y(t)=& \frak Cx(t) \nonumber
\end{align}
with input $v$ and output $y$
is a well-posed  port-Hamiltonian system.
\end{proposition}

\begin{definition}The well-posed port-Hamiltonian system \eqref{eqn:pdeio} is {\em exactly controllable}, if there exists a 
time $\tau >0$ such 
that for all $x_1 \in X$ there
exists a control function $u\in L^2((0,\tau);\mathbb{C}^n)$ such that the corresponding mild solution satisfies $x(0)=0$ 
and 
$x(\tau)=x_1$.
\end{definition}

\begin{proposition}\cite[c.f. Remark 6.9]{We94}\label{prop:contr}
Let $(\frak A, \frak B, \frak C)$ be a well-posed  port-Hamiltonian system  \eqref{eqn:pdeio}. Assume that
$F$ is an admissible feedback operator. 
Then the closed-loop system $(\frak A,  (\frak B-F\frak C), \frak C)$ is exactly controllable if and only if the open-loop 
system $(\frak A, \frak B, \frak C)$  is 
exactly controllable.
\end{proposition}

\section{Exact controllability for port-Hamiltonian systems}

This section is devoted to the main result of this paper, that is, we show that every well-posed port-Hamiltonian system 
\eqref{eqn:pde} is exactly controllable.

Exact controllability for impedance energy preserving port-Hamiltonian system has been studied in \cite{JacZwa19}. 

\begin{proposition}\cite[Corollary 10.7]{JacZwa19} \label{gamm} An  impedance 
energy  
preserving port-Hamiltonian system \eqref{eqn:pdeio} is exactly controllable.
\end{proposition}

For completeness we include the proof of Proposition \ref{gamm}.

\begin{proof} As the port-Hamiltonian system \eqref{eqn:pdeio} is  impedance energy preserving the corresponding 
operator $A$
generates 
a unitary strongly continuous group. Thus, $-A$ generates a bounded strongly continuous semigroup and exact controllability 
is equivalent to 
optimizability, \cite[Corollary 2.2]{ReWe}. The system is called {\em optimizable} if for all $x_0 \in X$ there exists a 
control function $u \in L^2((0,\infty); \mathbb{C}^n)$ such that the 
corresponding mild solution $x$ satisfies $x \in L^2((0,\infty);X)$. Thus it is sufficient to show that the 
port-Hamiltonian system \eqref{eqn:pdeio} is optimizable.
Let $x_0\in X$ be arbitrarily. In \cite[Lemma 7]{HuPa} it is shown that for every $k>0$ the choice $u(t)=-ky(t)$ leads to a 
mild solution in $L^2((0,\infty);X)$.
This shows 
optimizability of system \eqref{eqn:pdeio} and concludes the proof.
\end{proof}

Now we can formulate our main result.
\begin{theorem}\label{control}
Every well-posed port-Hamiltonian system \eqref{eqn:pde} is exactly controllable.
\end{theorem}

For the proof of our main result we need the following lemmas.
\begin{lemma}\label{Right}
 Let $\begin{bsmallmatrix}W_1 & W_0 \end{bsmallmatrix} \in \mathbb{C}^{n\times 2n}$ have full row rank with $W_1, W_0 \in 
\mathbb{C}^{n\times 
n}$. 
Then, there exist invertible matrices $\tilde R_1, \tilde R_0\in  \mathbb{C}^{n\times n}$
such that $\begin{bsmallmatrix}W_1 & W_0 \end{bsmallmatrix}\begin{bsmallmatrix}\tilde R_1 \\ \tilde R_0 
\end{bsmallmatrix}=I$.
\end{lemma}

\begin{proof}
Let $\begin{bsmallmatrix}W_1 & W_0 \end{bsmallmatrix}$   have full row rank with $\mathrm{rank\,} W_1=n-k$, $k\in 
\{0,\ldots,n\}$, and $\mathrm{rank\,} 
W_0=n-\ell$ with $\ell\in \{0,\ldots,n\}$.  Clearly $n-k+n-\ell\ge n$, or equivalently, $k+\ell\le n$.

By $W_1^{n-k}$ we denote the first $n-k$ rows of $W_1$ and $W_1^k$ denotes the last $k$ 
rows. Similarly, by $W_0^{n-\ell}$ we denote the last $n-\ell$ rows of $W_0$ and by $W_0^\ell$ the first $\ell$ rows. That 
is 
\begin{equation*}
W_1=\begin{bmatrix}W_1^{n-k} \\ W_1^{k}\end{bmatrix}\qquad \mbox{ and }\qquad W_0=\begin{bmatrix}W_0^{\ell} \\ 
W_0^{n-\ell}\end{bmatrix}. 
\end{equation*}
Without loss of generality, using row reduction and the fact that 
$\mathrm{rank\,}\begin{bsmallmatrix}W_1 & W_0 \end{bsmallmatrix}=n$, we may assume that $W_1^k=0$ and that $W^{n-k}_1$ and 
$W_0^{n-\ell}$ 
have full row rank.

We choose right inverses $R_1^{n-k} \in \mathbb{C}^{n\times (n-k)}$ for $W_1^{n-k}$ and $R_0^{n-\ell}\in 
\mathbb{C}^{n\times (n-\ell)}$ 
for $W_0^{n-\ell}$. Thus, 
\begin{equation*}
W_1^{n-k}R_1^{n-k}=I \qquad \mbox{ and }\qquad W_0^{n-\ell}R_0^{n-\ell}=I. 
\end{equation*}
Clearly, the columns of $R_1^{n-k}$ and $R_0^{n-\ell}$ are 
linearly independent and are not elements of the kernel of $W_1$ and $W_0$, respectively.

  Let $R_1^{k}\in \mathbb{C}^{n\times k}$ 
consisting of columns spanning the kernel of $W_1$, and let $R_0^{\ell}\in \mathbb{C}^{n \times \ell}$ consisting of 
columns 
spanning 
the kernel of $W_0$.
We define $R_1=\begin{bmatrix}R_1^{n-k} & R_1^k\end{bmatrix}\in \mathbb{C}^{n\times n}$ and $R_0=\begin{bmatrix}R_0^{\ell} 
& 
R_0^{n-\ell}\end{bmatrix}\in \mathbb{C}^{n\times n}$.
Thus, $R_1$ and $R_0$ are
invertible and it yields
\begin{align*}
&W_1R_1+W_0R_0&\\
&=\begin{bmatrix}I_{n-k}& 0_{(n-k)\times k}\\ 0_{k\times (n-k)}& 
0_{k\times k}\end{bmatrix}+\begin{bmatrix} 0_{\ell\times \ell} & W_0^lR_0^{n-\ell} \\ 0_{(n-\ell)\times \ell} & I_{n-\ell} 
\end{bmatrix}.
\end{align*}
Thus, $W_1R_1+W_0R_0:=M$ is invertible as an upper triangular matrix and we define $\tilde R_1:=R_1M^{-1}$ and 
$\tilde 
R_0:=R_0M^{-1}$ to obtain the assertion of the lemma.
\end{proof}

\begin{lemma}\label{alpha}
Let $\alpha \neq 0$ and $(\frak A, \frak B)$ be a well-posed port-Hamiltonian system. Then the port-Hamiltonian system 
$(\frak A, \alpha\frak B)$ is well-posed as well. Moreover,  the system $(\frak A, \frak B)$ is exactly controllable if and 
only if the 
system $(\frak A, \alpha\frak B)$ is exactly controllable.
\end{lemma}
\begin{proof} Well-posed of the scaled system follows immediately.
The controllability of the two systems is equivalent, since we can scale the input function $u$ of  one system by 
$\alpha$ or $\frac{1}{\alpha}$ to get an input for the other system without changing the mild solution.
\end{proof} 

{\em Proof of Theorem \ref{control}:}
We start with an arbitrary port-Hamiltonian system \eqref{eqn:pde} described by the tuple $(\frak A, \frak B)$.

By Lemma \ref{alpha}, this system is exactly 
controllable if and only if for some $\alpha>0$ the system $(\frak A, \alpha\frak B)$
 is exactly 
controllable. We aim to prove that there exists an $\alpha >0$ such that the system $(\frak A, \alpha\frak B)$ 
is exactly controllable.

Thus, we aim to write the system  $(\frak A, \alpha\frak B)$ as a closed-loop system of an exactly controllable system  
$(\frak A, \frak{B}_o,\frak{C}_o)$. To construct $(\frak A, \frak{B}_o,\frak{C}_o)$ we find an impedance energy 
preserving system $(\frak A, \frak{B}_o,\frak{\tilde C})$  which is exactly controllable by Proposition \ref{gamm}.

 By \eqref{operatorB} and \eqref{operatorB2}, the operator $\frak B$ is described by a full row rank $n\times 2n$-matrix 
 \begin{align*}
 W_B=\begin{bmatrix} W_1 &W_0\end{bmatrix}.
 \end{align*}
 Using Lemma \ref{Right} there exists a  matrix $R=\left[\begin{smallmatrix} R_1 \\ R_0 \end{smallmatrix} \right]\in 
\mathbb{C}^{2n \times n}$ such 
that 
\begin{align*}
W_B R=I
\end{align*}
 and $R_1, R_0 \in \mathbb{C}^{n \times n}$ are invertible. If $W_0=0$, without loss of generality we may assume that 
$R_0=I$ and 
$R_1=W_1^{-1}$.

We now consider the port-Hamiltonian system $(\frak A,{\frak B}_o, \widetilde{\frak C})$, where
\begin{align*}
{\frak B}_ox=
\begin{bmatrix}R_1^{-1} & 0 \end{bmatrix}\begin{bmatrix}f_{\delta,\mathcal{H} x}\\ e_{\delta,\mathcal{H} x}\end{bmatrix}
\end{align*}
and
\begin{align*}
\widetilde{\frak C}x=\begin{bmatrix} 0& R_1^{*} \end{bmatrix}\begin{bmatrix}f_{\delta,\mathcal{H} x}\\ 
e_{\delta,\mathcal{H} x}\end{bmatrix}.
\end{align*}

Obviously, the port-Hamiltonian  system $(\frak A,{\frak B}_o, \widetilde{\frak C})$ is impedance energy preserving. Then 
it 
follows from Proposition \ref{gamm}
that $(\frak A,{\frak B}_o, \widetilde{\frak C})$ is exactly controllable.

If $W_0=0$, then $(\frak A,{\frak B})=(\frak A,{\frak B}_o)$ and thus the statement is proved with $\alpha=1$.

We now assume that $W_0\not=0$.
In this case we consider the port-Hamiltonian system $(\frak A,{\frak B}_o, {\frak C}_o)$, where
\begin{align*}
{\frak C}_ox=
\begin{bmatrix}\alpha R_1^{-1} &\alpha R_0^{-1} \end{bmatrix}\begin{bmatrix}f_{\delta,\mathcal{H} x}\\ 
e_{\delta,\mathcal{H} x}\end{bmatrix}.
\end{align*}
The constant $\alpha>0$ will be chosen later.
The 
matrix $\begin{bsmallmatrix}R_1^{-1} & 0 \\\alpha R_1^{-1} &\alpha R_0^{-1} \end{bsmallmatrix}$ is invertible and the 
port-Hamiltonian system $(\frak A,{\frak B}_o,{\frak C}_o)$
 is still exactly controllable, since changing the output does not influence controllability.
 
 The port-Hamiltonian system $(\frak A,{\frak B}_o, {\frak C}_o)$ is regular, see Theorem \ref{wellposed} and Lemma 
\ref{lem:regular}. By $D$ we denote the feedthrough operator of $(\frak A,{\frak B}_o, {\frak C}_o)$ and 
 we choose 
 \begin{align*}
 \alpha=\begin{cases} 2 \left\|{D}\right\| \left\| {W_0R_0}\right\|, & D\not=0\\
 1, & D=0 \end{cases}.
 \end{align*}
 Then $\alpha>0$ and the matrix
 \begin{align*}F=\frac{1}{\alpha}W_0R_0
 \end{align*}
 is an admissible feedback operator for $(\frak A,{\frak B}_o, {\frak C}_o)$ as $\|DF\|<1$ (which implies invertibility of 
$I-DF$).
 
 We now consider the closed-loop system as shown in Figure \ref{fig} and 
obtain
\begin{align*}
   \dot{x} (t) &= \frak A  x(t),\qquad   x(0)= x_0, \\
   u_\alpha(t) &= \alpha  ( u_o(t) - F y_o(t))\\
   &= \alpha ({\frak B}_o-F{\frak C}_o)x(t)\\
   &= \left( \alpha \begin{bmatrix}R_1^{-1} & 0 \end{bmatrix} -W_0R_0\begin{bmatrix}\alpha R_1^{-1} &\alpha R_0^{-1} 
\end{bmatrix}\right)\begin{bmatrix}f_{\delta,\mathcal{H} x}\\ e_{\delta,\mathcal{H} x}\end{bmatrix}\\
   &=\alpha   W_B \begin{bmatrix}f_{\delta,\mathcal{H} x}\\ e_{\delta,\mathcal{H} x}\end{bmatrix}.
\end{align*}
Thus, the closed-loop system equals the port-Hamiltonian system $(\frak A, \alpha\frak B)$. As the open-loop system $(\frak 
A,{\frak B}_o, {\frak C}_o)$ is exactly controllable, by Theorem \ref{prop:contr}
the port-Hamiltonian system $(\frak A, \alpha\frak B)$ is exactly controllable. 

\tikzstyle{block} = [draw,  rectangle, 
    minimum height=3em, minimum width=6em]
    \tikzstyle{blocksmall} = [draw,  rectangle, 
    minimum height=1.5em, minimum width=3em]
\tikzstyle{sum} = [draw,  circle, node distance=2cm]
\tikzstyle{input} = [coordinate]
\tikzstyle{output} = [coordinate]
\tikzstyle{pinstyle} = [pin edge={to-,thin,black}]

\begin{figure}[ht]
    \begin{center}

\begin{tikzpicture}[auto, node distance=1.8cm,>=latex']
    \node [input, name=input] {};
		\node [blocksmall, right of=input] (op){$\frac{1}{\alpha}$};
    \node [sum, right of=op] (sum) {};
    \node [block, right of=sum] (system) {$(\frak A,{\frak B}_o, {\frak C}_o)$};
    \node [output, right of=system] (output) {};
    \node [block, below of=system] (feedback) {$F=\frac{1}{\alpha} W_0R_0$};
    \draw [draw,->] (input) -- node {$u_{\alpha}$} (op);
		\draw [->] (op) -- node {} (sum);
		\draw [->] (sum) -- node {$u_o$} (system);
				\draw [-] (system) -- node {$y_o$} (output);
					    \draw [->] (output) |- node[pos=0.99] {} 
        node [near end] {} (feedback);
    \draw [->] (feedback) -| node[pos=0.99] {$+$} 
        node [near end] {} (sum);
\end{tikzpicture}
    \end{center}
    \caption{$(\frak A, \alpha\frak B)$ as a closed-loop system}\label{fig}
\end{figure}
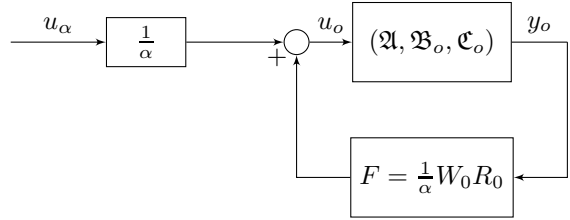
Thus, every well-posed port-Hamiltonian system is exactly controllable.
\hfill{\small$\square$}

\section{Example of an exactly controllable port-Hamiltonian system}

An (undamped) vibrating string can be modeled by 
 \begin{align}\label{wave}
    \frac{\partial^2 w}{\partial t^2} (\zeta,t) = \frac{1}{\rho(\zeta)} \frac{\partial }{\partial \zeta} \left( T(\zeta) 
\frac{\partial w}{\partial \zeta}(\zeta,t) \right), 
 \end{align}
 $ t\ge 0$,  $\zeta\in (0,1)$,
 where $\zeta\in [0,1]$ is the spatial variable, $w(\zeta,t)$
  is the vertical position of the string at place $\zeta$ and time $t$, $T (\zeta)>0 $ is the Young's modulus
  of the string, and $\rho (\zeta ) >0$ is the mass density, which may vary along the string. We assume that $T$ and $\rho$ 
are positive and continuously differentiable functions on $[0,1]$. By choosing the state variables 
$x_1= \rho \frac{\partial w}{\partial t}$ (momentum) and $x_2 = \frac{\partial w}{\partial \zeta}$ (strain), the partial 
differential equation can equivalently be written as
\begin{align}
  \nonumber
  \frac{\partial }{\partial t} \begin{bmatrix} x_1(\zeta,t) \\ x_2(\zeta,t) \end{bmatrix} &= \begin{bmatrix} 0 & 1 \\ 1 & 0 
\end{bmatrix} \frac{\partial }{\partial \zeta}\left( \begin{bmatrix} \frac{1}{\rho(\zeta)} & 0 \\ 0 & T(\zeta) 
\end{bmatrix}\begin{bmatrix} x_1(\zeta,t) \\ x_2(\zeta,t) \end{bmatrix} \right) \\
 \label{eq:7.2.3} &= P_1  \frac{\partial }{\partial \zeta}\left( {\cal H}(\zeta)\begin{bmatrix} x_1(\zeta,t) \\ x_2(\zeta,t) 
\end{bmatrix} \right),
\end{align}
where 
\begin{align*} P_1=\begin{bmatrix} 0 & 1\\ 1 & 0 \end{bmatrix}, \quad    {\cal 
H}(\zeta)=\begin{bmatrix}\frac{1}{\rho(\zeta)} & 0\\ 0 & T(\zeta)  \end{bmatrix}.
\end{align*}

 The boundary control for (\ref{eq:7.2.3}) is given by 
\[ \begin{bmatrix} \widetilde W_1 &  \widetilde W_0 \end{bmatrix} \begin{bmatrix}  ({\mathcal H}x)(1,t) \\({\mathcal 
H}x)(0,t) \end{bmatrix} =u(t),\]
where $\begin{bmatrix}\widetilde W_1 & \widetilde W_0 \end{bmatrix}$ is a $2\times 4$-matrix with rank 2,
or equivalently, the partial differential equation is equipped with the boundary control  
\begin{align}\label{wave2}
  \begin{bmatrix}  \widetilde W_1 &  \widetilde W_0 \end{bmatrix}  \begin{bmatrix}   \rho \frac{\partial w}{\partial t} 
(1,t)\\\frac{\partial w}{\partial \zeta} (1,t)\\ \rho \frac{\partial w}{\partial t} (0,t) \\\frac{\partial w}{\partial 
\zeta}(0,t)\end{bmatrix} =u(t).
 \end{align} 
Defining $\gamma =\sqrt{ T(\zeta)/\rho ( \zeta) }$,  the matrix function $P_1 {\cal H}$ can be factorized as
\[    P_1 {\cal H} = \underbrace{\begin{bmatrix} \gamma & -\gamma \\ \rho^{-1} & \rho^{-1} \end{bmatrix}}_{S^{-1}} 
\underbrace{\begin{bmatrix} \gamma & 0 \\ 0 & -\gamma \end{bmatrix}}_\Delta
\underbrace{\begin{bmatrix} (2\gamma)^{-1} & \rho/2 \\  (2\gamma)^{-1} & \rho/2 \end{bmatrix}}_S. \]
In \cite{JacMorZwa15} it is shown that the port-Hamiltonian system \eqref{wave}, \eqref{wave2} is well-posed if and only if 
\[ \widetilde W_1 \begin{bmatrix} \gamma(1) \\ T(1) \end{bmatrix} \oplus  \widetilde W_0 \begin{bmatrix} -\gamma(0) \\ T(0) 
\end{bmatrix} =\mathbb C^2,\]
or equivalently if the vectors $ \widetilde W_1 \left[\begin{smallmatrix} \gamma(1) \\ T(1) \end{smallmatrix}\right]$ and $ 
\widetilde W_0 \left[\begin{smallmatrix} -\gamma(0) \\ T(0) \end{smallmatrix}\right]$ are linearly independent. 

By Theorem \ref{control} the port-Hamiltonian system \eqref{wave}, \eqref{wave2} is exactly controllable if the vectors $ 
\widetilde W_1 \left[\begin{smallmatrix} \gamma(1) \\ T(1) \end{smallmatrix}\right]$ and $ \widetilde W_0 
\left[\begin{smallmatrix} -\gamma(0) \\ T(0) \end{smallmatrix}\right]$ are linearly independent. 
Here we consider $ \widetilde W_1:= I$ and $ \widetilde W_0:= \left[\begin{smallmatrix} -1 & 0\\ 0&1 
\end{smallmatrix}\right]$.
Then the port-Hamiltonian system \eqref{wave}, \eqref{wave2} is exactly controllable  if  the vectors  $ 
\left[\begin{smallmatrix} \gamma(1) \\ T(1) \end{smallmatrix}\right]$ and $ \left[\begin{smallmatrix} \gamma(0) \\ T(0) 
\end{smallmatrix}\right]$ are linearly independent. 

\section{Conclusions}

In this paper we have studied the notion of exact controllability for a class of linear port-Hamiltonian system on a one 
dimensional spacial domain with full boundary control and no internal damping. 
We showed that for this class well-posedness implies exact controllability. Further, we applied the obtained results to the 
wave equation. 
\enlargethispage{-3.2in}

By duality a well-posed  port-Hamiltonian system $(\frak A, \frak B,\frak C)$ with state space $L^2((0,\infty);\mathbb 
C^n)$ 
and output space $\mathbb C^n$ is exactly observable.
An interesting problem for future research  is the characterization of exact 
controllability for port-Hamiltonian systems with internal damping, i.e, port-Hamiltonian systems where $P_0$ is not 
necessarily  skew-adjoint. We note, that the condition that $\widetilde W_B$ has full rank cannot be neglected, as in 
general without full boundary control a port-Hamiltonian system is not exact controllable. Another open question is the 
characterization of exact 
controllability for port-Hamiltonian systems of higher order, see \cite{Villegas_2007}. However, for these systems even the 
characterization of well-posedness is an open problem.

%%%%%%%%%%%%%%%%%%%%%%%%%%%%%%%%%%%%%%%%%%%%%%%%%%%%%%%%%%%%%%%%%%%%%%%%%%%%%%%%
\bibliographystyle{IEEEtran}
\bibliography{literaturabb}

\end{document}